\documentclass[12pt,reqno]{amsart}

\usepackage{amssymb}

\newtheorem{theorem}{Theorem}[section]
\newtheorem{proposition}[theorem]{Proposition}
\newtheorem{lemma}[theorem]{Lemma}

\numberwithin{equation}{section}

\begin{document}

\baselineskip=15.5pt

\title[Brauer group of the moduli spaces of stable vector bundles]{Brauer group of
the moduli spaces of stable vector bundles of fixed determinant over a smooth curve}

\author[I. Biswas]{Indranil Biswas}

\address{School of Mathematics, Tata Institute of Fundamental
Research, Homi Bhabha Road, Mumbai 400005, India}

\email{indranil@math.tifr.res.in}

\author[T. Sengupta]{Tathagata Sengupta}

\address{Department of Mathematics, University of Hyderabad, C. R. Rao Road,
Hyderabad 500046, India}

\email{tsengupta@gmail.com}

\subjclass[2010]{14H60, 14F22, 14D23, 53C08}

\keywords{Moduli space, Brauer group, $\mu_r$-gerbe, Picard group.}

\date{}

\begin{abstract}
Let $X$ be an irreducible smooth projective curve, defined over an algebraically closed 
field $k$, of genus at least three and $L$ a line bundle on $X$. Let ${\mathcal 
M}_X(r,L)$ be the moduli space of stable vector bundles on $X$ of rank $r$ and 
determinant $L$ with $r\, \geq\, 2$. We prove that the Brauer group ${\rm 
Br}(\mathcal{M}_X(r,L))$ is cyclic of order ${\rm g.c.d.}(r,{\rm degree}(L))$. We also 
prove that ${\rm Br}(\mathcal{M}_X(r,L))$ is generated by the class of the projective 
bundle obtained by restricting the universal projective bundle. These results were 
proved earlier in \cite{BBGN} under the assumption that $k\,=\,\mathbb C$.
\end{abstract}

\maketitle

\section{Introduction}

Let $X$ be a compact connected Riemann surface of genus $g$, with $g\, \geq\, 3$. 
Fix a holomorphic line bundle $L$ over $X$ and also fix an integer $r\, \geq\, 2$. 
Let ${\mathcal M}_X(r,L)$ denote the moduli space of stable vector bundles on $X$ of 
rank $r$ and determinant $L$, which is a smooth quasiprojective complex variety of 
dimension $(r^2-1)(g-1)$. There is a Poincar\'e vector bundle over $X\times 
{\mathcal M}_X(r,L)$ if and only if $r$ and $\text{degree}(L)$ are coprime 
\cite{Ra}. When $r$ and $\text{degree}(L)$ are coprime, any two Poincar\'e vector 
bundle over $X\times {\mathcal M}_X(r,L)$ differ by tensoring with a line bundle 
pulled back from ${\mathcal M}_X(r,L)$. Hence the projectivized Poincar\'e bundle in 
unique. Even when $r$ and $\text{degree}(L)$ are not coprime, there is a unique 
projective Poincar\'e bundle over $X\times {\mathcal M}_X(r,L)$, although it is not 
a projectivization of a vector bundle.

In \cite{BBGN} it was proved that the Brauer group of ${\mathcal M}_X(r,L)$ is 
cyclic of order ${\rm g.c.d.}(r,{\rm degree}(L))$. As mentioned above, there is a 
universal projective bundle $\mathcal P$ on $X\times {\mathcal M}_X(r,L)$. Fixing a 
point $x\, \in\, X$, let ${\mathcal P}_x$ be the projective bundle on ${\mathcal 
M}_X(r,L)$ obtained by restricting $\mathcal P$ to $\{x\}\times {\mathcal 
M}_X(r,L)$. In \cite{BBGN} it was also shown that the Brauer group 
$\text{Br}({\mathcal M}_X(r,L))$ is generated by the class of ${\mathcal P}_x$.

Our aim here is to prove these results for all algebraically closed fields;
see Theorem \ref{thm1}.

The computation in \cite{BBGN} crucially uses the calculation of the Picard group of 
${\mathcal M}_X(r,L)$. It may be mentioned that the assumption in \cite{DN} that the 
characteristic of the base field is zero is used in the computation of the Picard 
group of the moduli space ${\mathcal M}_X(r,L)$. In particular, the Reynolds' 
operators, which play a crucial role in the computation, are valid only in 
characteristic zero. A recent theorem of Hoffmann shows that the Picard group of the 
moduli space does not depend on the base field \cite{Ho}. The proof of
Theorem \ref{thm1} follows the strategy of \cite{BBGN}; some details not
given in \cite{BBGN} are given here.

\section{Universal projective bundle and Brauer group}

Let $k$ be an algebraically closed field. Let $X$ be an irreducible smooth
projective curve, defined over $k$, of genus $g$, with $g\, \geq\, 3$. Fix
an integer $r\, \geq\, 2$ and also fix a line bundle $L$ over $X$. The degree
of $L$ will be denoted by $d$. Let ${\mathcal M}_X(r,d)$ be the
moduli space of stable vector bundles on $X$ of rank $r$ and degree $d$.
Consider the morphism
$$
\phi\, :\, {\mathcal M}_X(r,d)\, \longrightarrow\, \text{Pic}^d(X)\, ,
\ \ \ E\, \longmapsto\, \bigwedge\nolimits^r E\, .
$$
Let
$$
{\mathcal M}_X\,=\, {\mathcal M}_X(r,L)\, :=\, \phi^{-1}(L)
$$
be the fiber of $\phi$ over the point $L\, \in\, \text{Pic}^d(X)$. This moduli
space ${\mathcal M}_X$ is canonically identified with the following two moduli
spaces:
\begin{enumerate}
\item moduli space of pairs of the form $(E,\, \xi)$, where $E$ is a
stable vector bundle over $X$ of rank $r$, and $\xi\, :\, \bigwedge^r E
\, \longrightarrow\, L$ is an isomorphism, and

\item the moduli space of stable vector bundles $E$ on $X$ of rank $r$ such that
$\bigwedge^r E$ is isomorphic to $L$
\end{enumerate}
(see \cite[p.~1308, Proposition 2.1]{Ho}).

It is known that there is a universal projective bundle
\begin{equation}\label{bp}
{\mathcal P}\, \longrightarrow\, X\times {\mathcal M}_X
\end{equation}
(\cite{Ra}, \cite{Ne}).
This follows from the construction of the moduli space and the fact that
the global automorphisms of a stable vector bundle are nonzero constant
scalar multiplications; this projective bundle $\mathcal P$ is described in the proof of
Proposition \ref{prop1}. Fix a closed point $x\, \in\, X$. Let
\begin{equation}\label{e1}
{\mathcal P}_x\,:=\, {\mathcal P}\vert_{\{x\}\times {\mathcal M}_X}
\, \stackrel{f}{\longrightarrow}\, {\mathcal M}_X
\end{equation}
be the restriction of ${\mathcal P}$ to $\{x\}\times {\mathcal M}_X$.

For any quasiprojective variety $Y$ defined over the field $k$, the Brauer group 
$\text{Br}(Y)$ of $Y$ is defined to be the Morita equivalence classes of Azumaya 
algebras over the variety $Y$. It is known that this Brauer group $\text{Br}(Y)$ 
coincides with the equivalence classes of all principal $\text{PGL}_k$--bundles over 
$Y$, where two principal $\text{PGL}_k$--bundles $P$ and $Q$ are identified if there 
are two vector bundles $V_1$ and $V_2$ over $Y$ satisfying the condition that the two 
principal $\text{PGL}_k$--bundles $P\otimes {\mathbb P}(V_1)$ and $Q\otimes {\mathbb 
P}(V_2)$ are isomorphic. The addition of two projective bundles $P$ and $Q$ in the
Brauer group $\text{Br}(Y)$ is defined to be the equivalence class of the projective bundle
$P\otimes Q$. The inverse of a projective bundle $P$ in $\text{Br}(Y)$ is the
equivalence class of the dual projective bundle $P^*$. (See \cite{Gr}, \cite{Gr2},
\cite{Gr3}, \cite{Mi}, \cite{Ga} for properties of Brauer groups.)
The cohomological Brauer group $\text{Br}'(Y)$ of the variety 
$Y$ is the torsion part of the \'etale cohomology group $H^2_{et}(Y,\, {\mathbb G}_m)$. 
There is a natural injective homomorphism $\text{Br}(Y)\,\longrightarrow\, 
\text{Br}'(Y)$ which is in fact an isomorphism by a theorem of Gabber \cite{dJ}, 
\cite{Hoo}.

\begin{proposition}\label{prop1}
The Brauer group ${\rm Br}(\mathcal{M}_X)$ is generated by the
class $cl({\mathcal P}_x)\, \in\, {\rm Br}(\mathcal{M}_X)$ of the projective
bundle ${\mathcal P}_x$ defined in \eqref{e1}.
\end{proposition}

\begin{proof}
Given any line bundle $L_0$ on $X$, the morphism
$$
{\mathcal M}_X\,=\, {\mathcal M}_X(r,L)\,\longrightarrow\,
{\mathcal M}_X(r,L\otimes L^r_0)\, ,\ \ E\, \longmapsto\, E\otimes L_0
$$
is an isomorphism. The natural isomorphism of ${\mathbb P}(E\otimes L_0)$ with
${\mathbb P}(E)$ produces an isomorphism between the universal projective bundles
over $X\times {\mathcal M}_X(r,L)$ and $X\times {\mathcal M}_X(r,L\otimes L^r_0)$.
Therefore, after tensoring with a line bundle $L_0$ of sufficiently large
degree, we may assume that
$$
\frac{d}{r} \,>\, 2g-1\, .
$$
Let $\overline{\mathcal M}_X$ denote the moduli space of semistable vector bundles
$E$ on $X$ of rank $r$ with $\bigwedge^r E\,=\, L$.

The cotangent bundle of $X$ will be denoted by $K_X$.
For any vector bundle $E\, \in\, \overline{\mathcal M}_X$ and any point $y\, \in\, X$,
$$
H^1(Y,\, E\otimes {\mathcal O}_X(-y))\,=\,H^0(Y,\, E^*\otimes K_X
\otimes {\mathcal O}_X(y))^*\,=\, 0
$$
because ${\rm degree}(E^*\otimes K_X\otimes{\mathcal O}_X(y))\, <\, 0$
and $E^*\otimes K_X\otimes {\mathcal O}_X(y)$ is semistable. So from the
long exact sequence of cohomologies associated to the short exact sequence
$$
0\,\longrightarrow\,E\otimes {\mathcal O}_X(-y)\,\longrightarrow\, E
\,\longrightarrow\, E_y \,\longrightarrow\, 0
$$
it follows that the evaluation homomorphism $H^0(X,\, E)\,\longrightarrow\, E_y$
is surjective; hence $E$ is generated by its global sections.

Take any $E\,\in\, \overline{\mathcal M}_X$. Since the vector bundle $E$ is
generated by its global sections, there is a short exact sequence 
\begin{equation}\label{req}
0\,\longrightarrow\, \mathcal{O}^{\oplus (r-1)}_X\,\longrightarrow\, E\,
\longrightarrow\,\bigwedge\nolimits^r E\,=\,L\,\longrightarrow\, 0\, .
\end{equation}
This short exact sequence
does not split because $E$ is semistable and $\text{degree}(L)\, >\,0$.
All such nontrivial extensions are parameterized by
$$\mathbb{P}(H^1(X, {\rm Hom}(L,\mathcal{O}^{\oplus(r-1)}_X))^*)\,=\,
\mathbb{P}((H^1(X,L^*)^*\otimes_k k^{\oplus (r-1)}) \,=\, \mathbb{P}((H^1(X,L^*)^{
\oplus (r-1)})^*)\, .$$
The standard action of $\text{GL}(r-1,{k})$ on $k^{\oplus (r-1)}$ produces an action of
$\text{GL}(r-1,{k})$ on the projective space $\mathbb{P}((H^1(X,L^*)^{\oplus (r-1)})^*)$. The
moduli space $\overline{\mathcal M}_X$ is the geometric invariant theoretic quotient
$$
\mathbb{P}((H^1(X,L^*)^{r-1})^*)/\!\!/\text{GL}(r-1,{k})\,=\,
\mathbb{P}((H^1(X,L^*)^{r-1})^*)/\!\!/\text{PGL}(r-1,{k})\,=\,
\overline{\mathcal M}_X
$$
(see \cite{Ne}, \cite{DN}).

The tautological line bundle ${\mathcal O}_{\mathbb{P}((H^1(X,L^*)^{r-1})^*)}(1)$
on $\mathbb{P}((H^1(X,L^*)^{r-1})^*)$ will be denoted by ${\mathcal L}_0$. Let
$$
p_1\, :\, X\times\mathbb{P}((H^1(X,L^*)^{r-1})^*)\,\longrightarrow\,X\, ,\ \
p_2\, :\, X\times\mathbb{P}((H^1(X,L^*)^{r-1})^*)\,\longrightarrow\,
\mathbb{P}((H^1(X,L^*)^{r-1})^*)
$$
be the natural projections. There is a universal extension over
$X\times\mathbb{P}((H^1(X,L^*)^{r-1})^*)$
\begin{equation}\label{a1}
0 \,\longrightarrow\, (p^*_1 \mathcal{O}^{r-1}_X)\otimes p^*_2{\mathcal L}_0
\,\longrightarrow\, {\mathcal E} \,\longrightarrow\, p^*_1 L\,\longrightarrow\, 0\, .
\end{equation}
Let ${\mathcal U}\, \subset\, \mathbb{P}((H^1(X,L^*)^{r-1})^*)$ be the subset
defined by all points
$t\, \in\, \mathbb{P}((H^1(X,L^*)^{r-1})^*)$ such that the vector bundle ${\mathcal
E}\vert_{X\times\{t\}}$ is stable. This subset $\mathcal U$ is nonempty Zariski open. Let
\begin{equation}\label{t}
\theta\,:\, {\mathcal U}\,\longrightarrow\,{\mathcal U}/\!\!/ \text{PGL}(r-1,{k})\,=\,
{\mathcal M}_X
\end{equation}
be the quotient map. Consider the action of
$\text{PGL}(r-1,{k})$ on $X\times \mathbb{P}((H^1(X,L^*)^{r-1})^*)$ given by the
trivial action on $X$ and the above action of $\mathbb{P}((H^1(X,L^*)^{r-1})^*)$.
This action lifts to an action of
$\text{PGL}(r-1,{k})$ on ${\mathbb P}({\mathcal E})$. The corresponding
geometric invariant theoretic quotient
$$
{\mathcal P}\, :=\, 
({\mathbb P}({\mathcal E})\vert_{X\times{\mathcal U}})/\!\!/\text{PGL}(r-1,{k})
$$
is the universal projective bundle on $X\times {\mathcal M}_X$ (see \eqref{bp}).

Consider the map
\begin{equation}\label{F}
F\,:=\, \text{Id}_X\times f\, :\, X\times {\mathcal P}_x\,\longrightarrow\,
X\times{\mathcal M}_X\, ,
\end{equation}
where $f$ is the projection in \eqref{e1}. We will construct a vector bundle
$$
{\mathcal V} \,\longrightarrow\, X\times {\mathcal P}_x
$$
with the property that ${\mathbb P}({\mathcal V})\,=\, F^*{\mathcal P}$.

Let ${\mathcal E}_x\,:=\, {\mathcal E}\vert_{\{x\}\times {\mathcal 
U}}\,\longrightarrow\, {\mathcal U}$ be the vector bundle obtained by restricting 
$\mathcal E$ in \eqref{a1} to $\{x\}\times {\mathcal U}$. Let
$$
{\mathcal Q}\,:=\, {\mathbb 
P}({\mathcal E}_x)\, \stackrel{\beta'}{\longrightarrow}\, {\mathcal U}
$$ be the 
corresponding projective bundle. Define
$$
\beta\, :=\, \text{Id}_X\times\beta'\, :\, 
X\times {\mathcal Q} \,\longrightarrow\, X\times {\mathcal U}\, ,
$$
and consider the pulled back vector bundle
$$
\widetilde{\mathcal{E}}\,:=\, (\beta^*\mathcal{E})\otimes (q^*_2
{\mathcal O}_{\mathcal Q}(-1)) \,\longrightarrow\,X\times {\mathcal Q}\, ,
$$
where $q_2\, :\, X\times {\mathcal Q}\,\longrightarrow\,{\mathcal Q}$ is the natural
projection, and $${\mathcal O}_{\mathcal Q}(1)\,\longrightarrow\,
{\mathbb P}({\mathcal E}_x)\,=\,
{\mathcal Q}$$ is the tautological line bundle. For the natural action of $\text{GL}(r-1,
k)$ on $\widetilde{\mathcal{E}}$, the center ${\mathbb G}_m$ of $\text{GL}(r-1,{k})$
acts trivially on $\widetilde{\mathcal{E}}$. Consequently, the geometric invariant
theoretic quotient
$$
{\mathcal V} \,:=\, \widetilde{\mathcal{E}}/\!\!/\text{GL}(r-1,{k}) \,
{\longrightarrow}
\, X\times({\mathcal Q}/\!\!/\text{GL}(r-1,{k}))\,=\, X\times {\mathcal P}_x
$$
is a vector bundle. It is straight-forward to check that
\begin{itemize}
\item ${\mathbb P}({\mathcal V})\,=\, F^*{\mathcal P}$, where $F$ is the map in
\eqref{F}, and

\item for each point $y\, \in\,{\mathcal P}_x$, the vector bundle
${\mathcal V}\vert_{X\times\{y\}}$ on $X$ lies in the isomorphism class of
vector bundles associated to the point $f(y)\, \in\,{\mathcal M}_X$, where
$f$ is defined in \eqref{e1}.
\end{itemize}

Let $B\, :\, X\times {\mathcal P}_x\,\longrightarrow\,{\mathcal P}_x$
be the natural projection. 
Consider the direct image $B_*{\mathcal V}\, \longrightarrow\, {\mathcal P}_x$.
Let $Z\,\subset\, {\mathcal P}_x$ be a nonempty Zariski open subset
such that the restriction $(B_*{\mathcal V})\vert_Z$ is a trivial vector bundle.
Fix a trivialization of $(B_*{\mathcal V})\vert_Z$. Take a point $y_0\, \in\,
Z$ and choose $r-1$ linearly independent sections $$s_1,\, \cdots ,\, s_{r-1}\, \in\,
H^0(X\times\{y_0\},\, {\mathcal V}\vert_{X\times\{y_0\}})$$
such that the coherent subsheaf of ${\mathcal V}\vert_{X\times\{y_0\}}$ generated by
$s_1,\, \cdots ,\, s_{r-1}$ is a subbundle of ${\mathcal V}\vert_{X\times\{y_0\}}$
of rank $r-1$; we note that from \eqref{req} it follows immediately that such
$r-1$ linearly independent sections exist. Extend each $s_i$ to a section $\widetilde{s}_i$
of ${\mathcal V}\vert_{X\times Z}$ using the above trivialization of
$(B_*{\mathcal V})\vert_Z$. There is a Zariski open subset $Z'\, \subset\, Z$ containing
$y_0$ such that the coherent subsheaf of ${\mathcal V}\vert_{X\times Z}$ generated by
$\widetilde{s}_1,\, \cdots ,\, \widetilde{s}_{r-1}$ is a subbundle of
${\mathcal V}\vert_{X\times Z'}$ of rank $r-1$. Note that this subbundle
over $X\times Z'$ is trivial and a trivialization is given by the images
of $\widetilde{s}_1,\, \cdots ,\, \widetilde{s}_{r-1}$. Therefore, on $X\times Z'$, we have a
short exact sequence of vector bundles
$$
0 \,\longrightarrow\, {\mathcal O}^{\oplus (r-1)}_{X\times Z'}
\,\longrightarrow\, {\mathcal V}\vert_{X\times Z'} \,\longrightarrow\, L'
\,\longrightarrow\, 0\, ,
$$
where $L'$ is a line bundle on $X\times Z'$. Considering the top exterior products
it follows that for each point $y\, \in\, Z'$,
the restriction $L'\vert_{X\times\{y\}}$ is isomorphic to the line bundle $L$.
Now from the seesaw theorem (see \cite[p.~51, Corollary~6]{Mu}) it follows that
there is a line bundle $L''$ on $Z'$ such that the line bundle $L'\otimes B^*L''$ 
on $X\times Z'$ is isomorphic to the pullback of $L$ to $X\times Z'$. We may trivialize
$L''$ over suitable nonempty Zariski open subsets of $Z'$. Therefore, it follows 
that there is a nonempty Zariski open subset
\begin{equation}\label{i}
\iota\, :\, {\mathcal W}\, \hookrightarrow\, Z' \, \subset\, {\mathcal P}_x
\end{equation}
such that the restriction $L'\vert_{X\times\mathcal W}$ is isomorphic to
the pullback of $L$ to $X\times\mathcal W$.

Consequently, there is a morphism
$$
\varphi\, :\, {\mathcal W}\,\longrightarrow\, {\mathcal U}\, \subset\,
\mathbb{P}((H^1(X,L^*)^{r-1})^*)
$$
such that the following diagram is commutative
\begin{equation}\label{c}
\begin{matrix}
{\mathcal W}& \stackrel{\varphi}{\longrightarrow}& {\mathcal U}\\
~ \Big\downarrow \iota && ~\Big\downarrow\theta \\
{\mathcal P}_x& \stackrel{f}{\longrightarrow}& {\mathcal M}_X
\end{matrix}
\end{equation}
where $\theta$, $\iota$ and $f$ are the morphisms in \eqref{t}, \eqref{i} and
\eqref{e1} respectively.

The codimension of the complement $$\mathbb{P}((H^1(X,L^*)^{r-1})^*)\setminus
{\mathcal U}\, \subset\,\mathbb{P}((H^1(X,L^*)^{r-1})^*)$$ is at least two.
To prove this, note that $\text{Pic}({\mathcal U})\,=\, \mathbb Z$ \cite[p.~89,
Proposition 7.13]{DN} (here we need the assumption that $g\, \geq\, 3$); this immediately
implies that the codimension of the complement ${\mathcal U}^c$ is at least two. Since the
Brauer group of a projective space is zero, in view of this codimension estimate,
it follows from the ``Cohomological purity'' \cite[p.~241, Theorem~VI.5.1]{Mi} (it also
follows from \cite[p.~292--293]{Gr}) that
\begin{equation}\label{bz}
{\rm Br}({\mathcal U})\, =\, 0\, .
\end{equation}
{}From the commutativity of \eqref{c} we conclude that the pullback homomorphism
$$
\iota^*\circ f^* \, =\, (f\circ \iota)^* \, :\, \text{Br}({\mathcal M}_X)
\,\longrightarrow\, \text{Br}({\mathcal W})
$$
coincides with the homomorphism $\varphi^*\circ \theta^*\, :\, \text{Br}({\mathcal M}_X) 
\,\longrightarrow\, \text{Br}({\mathcal W})$. On the other hand, from \eqref{bz}
we know that $\varphi^*\circ \theta^*\,=\, 0$. Hence
$$
\iota^*\circ f^* \, =\,0\, .
$$

On the other hand, the homomorphism
$$
\iota^*\, :\, \text{Br}({\mathcal P}_x)
\,\longrightarrow\, \text{Br}({\mathcal W})
$$
is injective because $\mathcal W$ is a Zariski open dense subset of
${\mathcal P}_x$ (see \cite[p.~142, Theorem~2.5]{Mi}). Consequently,
$$
f^* \,: \,\text{Br}({\mathcal M}_X)
\,\longrightarrow\,\text{Br}({\mathcal P}_x)
$$
is the zero homomorphism. On the other hand, the kernel of the above
homomorphism $f^*$ is generated by the class of ${\mathcal P}_x$
\cite[p.~193, Theorem~2]{Ga}. Therefore, we conclude that $\text{Br}({\mathcal M}_X)$ is
generated by the class of ${\mathcal P}_x$. This completes the proof.
\end{proof}

We will denote the integer ${\rm g.c.d.}(r,d)$ by $\delta$.

\begin{lemma}\label{lem1}
The order of the Brauer class $cl({\mathcal P}_x)\, \in\, {\rm Br}(\mathcal{M}_X)$
is $\delta$.
\end{lemma}

\begin{proof}
Let $\widetilde{\mathcal M}_X\,=\, \widetilde{\mathcal M}_X(r,L)$ be the moduli stack
of pairs of the form $(E,\, \xi)$, where $E$ is a
stable vector bundle over $X$ of rank $r$ and $\xi\, :\, \bigwedge^r E
\, \longrightarrow\, L$ is an isomorphism. Let $\mu_r$ denote the kernel of the
homomorphism
$$
{\mathbb G}_m\, \longrightarrow\, {\mathbb G}_m\, ,\ \ z\, \longmapsto\, z^r\, .
$$
The natural morphism
$$
\gamma\, :\, \widetilde{\mathcal M}_X\,\longrightarrow\, {\mathcal M}_X
$$
makes $\widetilde{\mathcal M}_X$ a $\mu_r$--gerbe over ${\mathcal M}_X$. Let
$$
c_0\,\in\, H^2({\mathcal M}_X,\,\mu_r)
$$
be the class of this $\mu_r$--gerbe. The Brauer class
$$
cl({\mathcal P}_x)\, \in\, {\rm Br}(\mathcal{M}_X)\,=\, H^2({\mathcal M}_X,\,
{\mathbb G}_m)
$$
coincides with the image of $c_0$ under the homomorphism
$$
\eta\, :\, H^2({\mathcal M}_X,\,\mu_r)\, \longrightarrow\, H^2({\mathcal M}_X,\,
{\mathbb G}_m)
$$
given by the inclusion of $\mu_r$ in ${\mathbb G}_m$.

There is a short exact sequence
\begin{equation}\label{es1}
0\, \longrightarrow\, \text{Pic}(\mathcal{M}_X)
\, \stackrel{\gamma^*}{\longrightarrow}\, \text{Pic}(\widetilde{\mathcal M}_X)
\, \stackrel{\nu}{\longrightarrow}\, {\rm Hom}(\mu_r,\, {\mathbb G}_m)\,=\,
{\mathbb Z}/r{\mathbb Z} \, \stackrel{\alpha}{\longrightarrow}\,{\rm Br}(\mathcal{M}_X)
\end{equation}
\cite[p.~232, Lemma 4.4]{BH}, where $\nu$ sends a line bundle on the $\mu_r$--gerbe
$\widetilde{\mathcal M}_X$ to the weight associated
to the action of $\mu_r$ on it. The image
$$
\alpha(1)\, \in\, {\rm Br}(\mathcal{M}_X)\,=\, H^2({\mathcal M}_X,\,
{\mathbb G}_m)
$$
coincides with the image of the class of the $\mu_r$--gerbe $\widetilde{\mathcal M}_X$
under the above homomorphism $\eta$. From this it follows that
\begin{equation}\label{a}
\alpha(1)\,=\, cl({\mathcal P}_x)
\end{equation}
because $cl({\mathcal P}_x)$ also coincides with the image of the class of the
$\mu_r$--gerbe $\widetilde{\mathcal M}_X$ under the above homomorphism $\eta$.

{}From \cite[p.~1311, Lemma 3.6]{Ho}, \cite[p.~1311, Lemma 3.3]{Ho}
and \cite[p.~1310, Theorem 3.1]{Ho} it follows that
$$
{\mathbb Z}/\text{image}(\nu)\,=\, {\mathbb Z}/\delta{\mathbb Z}\, ,
$$
where $\nu$ is the homomorphism in \eqref{es1}. Therefore, from \eqref{es1} we conclude
that the order of $\alpha(1)\, \in\, {\rm Br}(\mathcal{M}_X)$ is $\delta$. Now
the lemma follows from \eqref{a}.
\end{proof}

Combining Proposition \ref{prop1} and Lemma \ref{lem1} we have:

\begin{theorem}\label{thm1}
The Brauer group ${\rm Br}(\mathcal{M}_X)$ is cyclic of order $\delta\,=\,
{\rm g.c.d.}(r,d)$. The group ${\rm Br}(\mathcal{M}_X)$ is generated by the
class $cl({\mathcal P}_x)\, \in\, {\rm Br}(\mathcal{M}_X)$
of the projective bundle ${\mathcal P}_x$ defined in \eqref{e1}.
\end{theorem}

\section*{Acknowledgements}

We thank the referee for helpful comments. The first-named author is supported by a 
J. C. Bose Fellowship. The second-named author would like to thank Tata Institute of 
Fundamental Research for its hospitality while the paper was being worked on.

%%%%%%%%%%%%%%%%%%%%%%%%%%%%%%%%%%%%%%%%%%%%%%%%%%%%%%%%%%%%%%

\end{document}